\documentclass{amsart}
\usepackage{graphicx} 
\usepackage{amssymb}
\usepackage{amsthm}
\usepackage{hyperref}
\usepackage{xcolor}
\usepackage{braket}
\usepackage[colorinlistoftodos]{todonotes}
\usepackage{comment}
\newtheorem{theorem}{Theorem}[section]
\newtheorem{lemma}[theorem]{Lemma}
\newtheorem{corollary}[theorem]{Corollary}

\theoremstyle{definition}
\newtheorem{definition}[theorem]{Definition}
\newtheorem{notation}[theorem]{Notation}
\newtheorem{example}[theorem]{Example}

\newtheorem{proposition}[theorem]{Proposition}
\newtheorem{fact}[theorem]{Fact}

\newtheorem{observation}[theorem]{Observation}

\theoremstyle{remark}
\newtheorem{remark}[theorem]{Remark}
\theoremstyle{remark}

\numberwithin{equation}{section}

\newcommand{\R}{\mathbb{R}}
\newcommand{\eg}{\operatorname{EG}}
\newcommand{\oo}{s}
\newcommand{\oi}{\ell}

\newcommand{\sym}{\operatorname{Sym}}
\newcommand{\fa}{\mathfrak{A}}
\newcommand{\aut}{\operatorname{Aut}}
\newcommand{\thy}{\operatorname{Th}}
\newcommand{\CC}{\mathcal{C}}
\newcommand{\PP}{\mathcal{P}}
\newcommand{\Pc}{\mathcal{P}}
\newcommand{\II}{\mathcal{I}}
\newcommand{\Lc}{\mathcal{L}}
\newcommand{\Oc}{\mathcal{O}}

\newcommand{\Afr}{\mathfrak{A}}
\newcommand{\Bfr}{\mathfrak{B}}
\newcommand{\Cfr}{\mathfrak{C}}

\newcommand{\abar}{\bar{a}}
\newcommand{\bbar}{\bar{b}}
\newcommand{\cbar}{\bar{c}}
\newcommand{\dbar}{\bar{d}}
\newcommand{\xbar}{\bar{x}}
\newcommand{\ybar}{\bar{y}}

\newcommand{\ZF}{\ifmmode\mathsf{ZF}\else$\mathsf{ZF}$\fi}
\newcommand{\ZFC}{\ifmmode\mathsf{ZFC}\else$\mathsf{ZFC}$\fi}
\newcommand{\DLO}{\ifmmode\mathsf{DLO}\else$\mathsf{DLO}$\fi}

\newcommand{\bracenom}{\genfrac{\lbrace}{\rbrace}{0pt}{}}

\title[Labelled growth rates and choiceless set theory]{Labelled growth rates of $\omega$-categorical structures and applications in choiceless set theory}
\author{Bertalan Bodor}
\address{Department for Discrete Mathematics and Geometry, FB1 Algebra Group, Technische Universit\"{a}t Wien, Wiedner Hauptstraße 8-10, 1040 Wien, Austria}
\email{bertalan.bodor@tuwien.ac.at}
\author{Samuel Braunfeld}
\address{Computer Science Institute, Charles University. Prague, Czech Republic; and The Czech Academy of Sciences, Institute of Computer Science, Pod Vod\'{a}renskou v\v{e}\v{z}\'{\i} 2, 182 00 Prague, Czech Republic.}
\email{sambraunfeld@gmail.com}
\author{James E. Hanson}
\address{Department of Mathematics, Iowa State University,  396 Carver Hall,  411 Morrill Road, Ames, IA 50011, USA}
\email{jameseh@iastate.edu}
\thanks{This paper is part of a project that has received funding from the European Research Council (ERC) under the European Union's Horizon 2020 research and innovation programme (grant agreement No 810115 - Dynasnet). 
The second author is further supported by Project 24-12591M of the Czech Science Foundation (GA\v{C}R), and supported partly by the long-term strategic development financing of the Institute of Computer Science (RVO: 67985807).
The first authour has been funded by the European Research Council (Project POCOCOP, ERC Synergy Grant 101071674). Views and opinions expressed are however those of the authors only and do not necessarily reflect those of the European Union or the European Research Council Executive Agency. Neither the European Union nor the granting authority can be held responsible for them.}

\begin{document}

\maketitle

\begin{abstract}
We study the labelled growth rate of an $\omega$-categorical structure $\fa$, i.e., the number of orbits of $Aut(\fa)$ on $n$-tuples of distinct elements, and show that the model-theoretic property of monadic stability yields a gap in the spectrum of allowable labelled growth rates. As a further application, we obtain gap in the spectrum of allowable labelled growth rates in hereditary graph classes, with no a priori assumption of $\omega$-categoricity. We also establish a way to translate results about labelled growth rates of $\omega$-categorical structures into combinatorial statements about sets with weak finiteness properties in the absence of the axiom of choice, and derive several results from this translation.
\end{abstract}

\section{Introduction}

\subsection{Counting in $\omega$-categorical structures}

In the 1970s, permutation group theorists, particularly Peter Cameron, began considering orbit-counting problems for groups acting on a countable set $X$, under the additional assumption that the number of orbits on $X^n$ is finite for every $n$. Under this assumption, we obtain sequences of natural numbers counting the number of orbits on unordered $n$-sets for every $n$, as well as on ordered $n$-tuples of distinct elements. Many of the problems focused on obtaining a detailed understanding of the situations with the most symmetry, i.e., where the sequences grow slowly, and identifying jumps in the realizable asymptotics of these sequences.

The finiteness condition on these problems means that, without loss of generality, we can restrict our attention to the natural action of the automorphism group of some countable $\omega$-categorical relational structure $M$. We may further assume $M$ is homogeneous, and so the number of orbits on unordered $n$-sets is equal to the number of $n$-substructures of $M$ up to isomorphism (unlabelled enumeration), while the number of orbits on $n$-tuples of distinct elements is equal to the number of $n$-substructures of $M$ with domain $[n]$ up to equality (labelled enumeration). The labelled problem can also be seen as counting the number of $n$-types of distinct elements of $M$. 

In the labelled setting, we prove the following theorem, using the model-theoretic properties of cellularity and monadic stability to identify jumps in the realizable asymptotics. Furthermore, Lachlan \cite{lachlan1992} and Bodor \cite{bodor2024classification} have given a structural classification of the $\omega$-categorical monadically stable structures, which the following theorem identifies as structures with sufficiently slow labelled growth.

\begin{theorem}[Theorem \ref{thm:main}]
    Let $\fa$ be an $\omega$-categorical structure. Then one of the following holds.
    \begin{enumerate}
        \item $\fa$ is cellular and there exist $c, d \in \R, d<1$ such that $\oi_n(\fa) \leq cn^{dn}$.
        \item $\fa$ is monadically stable but not cellular. In this case $\oi_n(\fa) \geq B_n$ where $B_n$ is the $n^{th}$ Bell number, and for sufficiently large $n$, $\oi_n(\fa) \leq n!/c^n$ for all $c > 0$.
        \item $\fa$ is not monadically stable and there is some polynomial $p(n)$ such that $\oi_n(\fa) \geq n!/p(n)$.
    \end{enumerate}
\end{theorem}

This is the labelled analogue of the main result of \cite{braunfeld2022monadic} in the unlabelled case, which identified monadic stability as explaining the gap between sub-exponential and exponential growth. In the labelled setting, linear orders now have somewhat fast growth and so the statement is cleaner: all structures with sufficiently slow growth must be monadically stable. This also sharpens the main result of \cite{bodirsky2021permutation}, which identified cellularity as explaining a jump, which we show truly reaches the Bell numbers. 

\subsection{Counting in hereditary graph classes}
There is another line of work, concerned with asymptotic labelled enumeration in \emph{hereditary} classes of graphs, i.e., classes closed under induced subgraph. Thus we have dropped the previous assumption that our hereditary class arises as the set of substructures of an $\omega$-categorical homogeneous structure. The concern with jumps and the structure of slow classes is similar, and previous work is surveyed in \cite{bollobas1997hereditary, bollobas12007hereditary}. In this setting, we identify a new gap by showing it is still true that any graph class with sufficiently slow growth must come from $\omega$-categorical monadically stable structures.

\begin{theorem}[Proposition \ref{prop:sd catms}, Theorem \ref{thm:graph gap}] \label{thm:intrographgap}
    Let $\CC$ be a hereditary graph class. Then one of the following holds.
    \begin{enumerate}
        \item $\CC$ has bounded shrub-depth and $\oi_n(\CC) < \frac{n!}{c^n}$ for every $c>0$, for sufficiently large $n$.
        \item $\CC$ has unbounded shrub-depth, and $\oi_n(\CC) > \frac{n!}{10^n}$, for sufficiently large $n$.
    \end{enumerate}
    Furthermore, $\CC$ has bounded shrub-depth if and only if it is the union of the ages of finitely many $\omega$-categorical monadically stable graphs.
\end{theorem}

Classes with labelled growth rate slower than $n^{(1+o(1))n}$ were previously well-understood \cite{balogh2000speed}. Furthermore, for every $c > 1$, there are classes whose labelled growth rate oscillates between $n^{(c+o(1))n}$ and $2^{n^{2-1/c}}$ infinitely often \cite[Theorem 10]{balogh2001penultimate}, so there can be no gaps in the possible growth rates in this range. But there may be significant structure within the $n^{(1+o(1))n}$ range, such as the gap above.

\subsection{Choiceless set theory and $\omega$-categoricity}

There is a strong connection between $\omega$-categorical first-order theories and certain kinds of `weakly finite' sets that can occur in set theory without the axiom of choice (i.e., sets satisfying conditions that are ordinarily equivalent to finiteness). In \cite{plotkin1969}, Plotkin showed how to generically embed countable models of $\omega$-categorical theories into models of \ZF.

\begin{fact}[Plotkin {\cite{plotkin1969}}]\label{fact:PlotkinEmbed}
  For any $\omega$-categorical theory $T$ in a countable language $\Lc$, there is a symmetric submodel $M$ of a forcing extension such that in $M$, there is an $\Lc$-structure $A$ whose underlying set is a subset of $\Pc(\Pc(\omega))$ such that $A \models T$ and for every $n$, every $X \subseteq A^n$ is first-order definable with parameters.
\end{fact}

This can be used to build many different kinds of failures of choice in \ZF. For instance, if $T$ is the theory of an infinite set, then every subset of $A$ will be finite or co-finite. If $T$ is \DLO{}, then $A$ will be a dense linear order with the property that every subset of $A$ is a finite union of intervals.

In \cite{walczaktypke2005}, Walczak-Typke showed (building on earlier work of Truss in \cite{Truss_1995}) a strong converse to Fact~\ref{fact:PlotkinEmbed} by mapping out a direct correspondence between various weak finiteness notions in \ZF{} (which we review in Section~\ref{sec:set-theory-preliminaries}) and model-theoretic properties. In Section~\ref{sec:growth-rates-in-weakly-finite-sets}, we will use this correspondence to prove some surprising combinatorial theorems about various kinds of weakly finite sets. Crucially though, we are able to avoid working in \ZF{} explicitly by using set-theoretic absoluteness. The general principle of this kind of argument is that in \ZF{}, given a well-ordered bundle of data $T$ (which in our case will always be a first-order theory in a well-orderable language), we can build a (proper class sized) inner model\footnote{$L[T]$ is G\"odel's universe of constructible sets relativized to $T$.} $L[T]$ of \ZFC, apply some standard argument there to prove something about $T$, and then note that the property in question is \emph{absolute} between $L[T]$ and $V$ (our original class of all sets), meaning that it is true when interpreted as a statement in $L[T]$ if and only if it is true when interpreted as a statement in $V$. For example, `$T$ is a consistent first-order theory' is an absolute statement, because inconsistency is witnessed by a finite object (some finite subset of $T$ together with a finite proof), but `$T$ is a countable set' is not absolute, since a bijection between $T$ and $\omega$ can exist in $V$ but fail to be an element of $L[T]$.

\section{Preliminaries} \label{sec:prelim}

\subsection{Growth rates and $\omega$-categoricity}

\begin{definition}
    A permutation group $G \leq Sym(X)$ is \emph{oligomorphic} if $G$ acting on $X^n$ has only finitely many orbits for every $n \in \omega$.\\

    A countable structure in a countable language $\fa$ is \emph{$\omega$-categorical} if $\aut(\fa)$ acting on $\fa$ is oligomorphic.
\end{definition}

For a set $A$ we write $A^{(n)}$ for the set of injective $n$-tuples from $A$, i.e., $$A^{(n)}=\{(a_1,\dots,a_n)\in A^n: | a_1, \dots\, a_n \text{ are distinct}\}.$$

\begin{definition}
	Let $\fa$ be an $\omega$-categorical structure. Then we denote by $\oi_n(\fa)$ (resp. $\oo_n(\fa)$) the number of orbits of $\aut(\fa)$ acting on $A^{(n)}$ (resp. $A^n$). We call the sequence $\oi_n(\fa)$ the \emph{labelled growth rate of $\fa$}.
\end{definition}

	Note that if $\fa$ is infinite then $\oi_n(\fa)\leq \oi_{n+1}(\fa)$. Indeed, if $\bar{a},\bar{b}\in A^{(n)}$ are in different $n$-orbits then we can pick some $c\in A$ not contained in either $\bar{a}$ or $\bar{b}$. Then $\bar{a}c,\bar{b}c\in A^{(n+1)}$ are also in different orbits.
	
\begin{observation}\label{stirling}
	The values of $\oo_n$ can be calculated by the formula $\oo_n(\fa)=\sum_{k=1}^n\bracenom{n}{k}\oi_k(\fa)$ where $\bracenom{n}{k}$ denotes the Stirling numbers  of the second kind.
\end{observation}

\begin{definition}
    Let $\CC$ be a hereditary class of relational structures. We let $\oi_n(\CC)$ denote the number of labelled structures in $\CC$ of size $n$, i.e., the number of structures in $\CC$ with vertex set $[n]$ counted up to equality rather than isomorphism. The \emph{labelled growth rate} of $\CC$ (also called the $\emph{speed}$) is the function $n \mapsto \oi_n(\CC)$.
\end{definition}

If $\fa$ is an $\omega$-categorical relational structure with $Age(\fa) = \CC$, then $\oi_n(\fa) \geq \oi_n(\CC)$, and equality is achieved if $\fa$ has quantifier elimination (which can be assumed without loss of generality when computing $\oi_n(\fa)$). This explains why we also refer to $\oi_n(\fa)$ as the labelled growth rate of $\fa$.

\subsection{Stability, monadic stability, NIP, and monadic NIP}

\begin{definition}
    A (possibly incomplete) theory $T$ is \emph{unstable} if there is some $M \models T$, a formula $\phi(\xbar; \ybar)$, and $(\abar_i \in M^{|\xbar|} : i \in \omega)$, $(\bbar_j \in M^{|\ybar|} : j \in \omega)$ such that $M \models \phi(\abar_i; \bbar_j) \iff i \leq j$. A theory $T$ is \emph{stable} if it is not unstable.

    A theory $T$ \emph{has the independence property} if there is some  $M \models T$, a formula $\phi(\xbar; \ybar)$, and $(\abar_i \in M^{|\xbar|} : i \in \omega)$, $(\bbar_j \in M^{|\ybar|} : j \in 2^{\omega})$ such that $M \models \phi(\abar_i; \bbar_j) \iff i \in j$. A theory $T$ is \emph{NIP} if it does not have the independence property.

        A theory $T$ is \emph{monadically stable} (resp., \emph{monadically NIP}) if every expansion of $T$ by any number of unary predicates is stable (resp., NIP). 
\end{definition}

We will say a structure $\fa$ has one of the properties above if $Th(\fa)$ does. And we will say a class of structures $\CC$ has one of the properties above if $\bigcap_{M \in \CC} Th(M)$ does.

It is useful to name the graph appearing in the definition of unstable theories.

\begin{definition}
    The \emph{half-graph} of order $n$ is the bipartite graph with parts $\set{a_1, \dots, a_n}$ and $\set{b_1, \dots, b_n}$ and an edge between $a_i$ and $b_j$ if and only if $i \leq j$.
\end{definition}

Canonical obstructions to monadic stability/NIP appear already in the theory $T$ itself, rather than needing to pass to a unary expansion. The following definition is essentially taken from~\cite{baldwin1985second}.
	
\begin{definition} \label{def:coding}
	A theory $T$ \emph{admits coding} if there exists a formula $\phi(x,y,z)$ with parameters, and infinite sets $A,B,C\subset M$ in some model $M$ of $T$ such that for all $(a,b)\in A\times B$, there exists $c\in C$ such that $\forall c'\in C(\phi(a,b,c')\Leftrightarrow c=c')$.

    A theory $T$ \emph{admits tuple-coding} if the we instead allow a formula $\phi(\xbar, \ybar, z)$, and similarly allow $A$ and $B$ to be sets of tuples of corresponding length.
\end{definition}

\begin{remark}
	Note that by compactness it is enough to require the above condition in the case when $A,B,C$ are arbitrarily large but finite. In this case these sets with the required property can be found in any model of $T$.
\end{remark}

\begin{theorem}[{\cite[Theorem 1.1]{braunfeld2021characterizations}}]
    For any theory $T$, the following are equivalent.
    \begin{itemize}
        \item $T$ is monadically NIP.
        \item $T$ does not admit tuple-coding.
    \end{itemize}
\end{theorem}

	The following result follows from Lemma 4.2.6 in~\cite{baldwin1985second}. Note however that this is not the exact formulation in the original paper. For more details see the discussion in~\cite{braunfeld2022monadic}, Section 2, as well as \cite{Anderson1990} and \cite[Fact 2.8]{braunfeld2022worst}.
	
\begin{theorem}\label{coding}
  For any stable theory $T$, the following are equivalent.
  \begin{itemize}
  \item $T$ is monadically stable.
  \item $T$ is monadically NIP.
  \item $T$ does not admit coding.
  \end{itemize}
\end{theorem}

\subsection{Weak finiteness notions in choiceless set theory}
\label{sec:set-theory-preliminaries}

He we will recall some weak finiteness notions from choiceless set theory and their relationship to certain model-theoretic adjectives. 

\begin{definition} Fix a set $A$.
  \begin{itemize}
  \item $A$ is \emph{Dedekind-finite} if it admits no injection from $\omega$.
  \item $A$ is \emph{power-Dedekind-finite}\footnote{This property is also called \emph{weak Dedekind-finiteness} in the literature.} if it admits no surjection onto $\omega$.
  \item $A$ is \emph{strictly Mostowski-finite} if any partial order on $A^n$ has chains of uniformly bounded length.
  \item $A$ is \emph{amorphous} if every subset of $A$ is either finite or co-finite.
  \end{itemize}
\end{definition}

It is relatively straightforward to show that any amorphous set is strictly Mostowski-finite, any strictly Mostowski-finite is power-Dedekind-finite, and any power-Dedekind-finite set is Dedekind-finite. Moreover, $A$ is power-Dedekind-finite if and only if $\Pc(A)$ is Dedekind-finite. Strict Mostowski-finiteness was introduced by Walczak-Typke in \cite{walczaktypke2005} in order to fit the correspondence in the following Fact~\ref{fact:WalczakTypkeCorrespondence}, but the other notions are classical. 

\begin{definition}
  Given a first-order theory $T$, we say that a set $A$ \emph{admits $T$} if there is a model of $T$ whose underlying set is $A$.
\end{definition}

\begin{fact}[Walczak-Typke {\cite{walczaktypke2005}}]\label{fact:WalczakTypkeCorrespondence}
  $(\ZF)$ Fix a set $A$.
  \begin{itemize}
  \item $A$ is power-Dedekind-finite if and only if every countable theory admitted by $A$ is $\omega$-categorical.
  \item $A$ is strictly Mostowski-finite if and only if every countable theory admitted by $A$ is $\omega$-categorical and NSOP.
  \item $A$ is amorphous if and only if every countable theory admitted by $A$ is $\omega$-categorical and strongly minimal.
  \end{itemize}
\end{fact}

Walczak-Typke also proves similar statements about other finiteness notions inspired by the close connection between power-Dedekind-finite sets in \ZF{} and model-theoretic properties, including MT-rankedness, introduced in \cite{mendick2003} by analogy with Morley rank and $\omega$-stability, and  o-amorphicity, introduced in \cite{creed2000} by analogy with o-minimality.

\section{Labelled growth rates}
\subsection{Monadically stable structures}
Our upper bound on the labelled growth rate of an $\omega$-categorical monadically stable structure will go through exponential generating functions, which we now recall.

\begin{definition}
    The \emph{exponential generating function} of a sequence $(a_n)_{n\in \omega}$ is defined to be the formal power series $\eg(a_n)(x)=\sum_{i=0}^\infty a_n\frac{x^n}{n!}$.
\end{definition}

	Note that if $EG(a_n)(x)$ converges for some $c$ then $\lim_{n\rightarrow \infty}a_nc^n/n!=0$, in particular $a_n<n!/c^n$ if $n$ is large enough.

\begin{definition}
	Let $G\leq \sym(X)$ and $H\leq \sym(Y)$ be permutation groups with $X\cap Y=\emptyset$.
\begin{itemize}
\item The \emph{direct product} of $G$ and $H$, denoted by $G\times H$, is the group of all permutations $\sigma\in \sym(X\cup Y)$ such that $\sigma$ preserves the subsets $X$ and $Y$, and $\sigma|_X\in G, \sigma|_Y\in H$.
\item The \emph{wreath product} of $G$ and $H$, denoted by $G\wr H$, is the the group of all permutations $\sigma \in \sym(X\times Y)$ such that there exists some $\beta\in H$ and $\alpha_y: y\in Y$ such that $\sigma(x,y)=(\alpha_y(x),\beta(y))$.
\end{itemize}
\end{definition}

We now describe how we may use these operations to inductively construct the automorphism group of any $\omega$-categorical monadically stable structure.

    \begin{fact}[{\cite[Theorem 3.42]{bodor2024classification}}] \label{mstab induct}
	A countable $\omega$-categorical structure $\fa$ is monadically stable if and only if $\aut(\fa)$, up to relabelling of the elements of the domain set, can be built up from finite domain permutation groups in finitely many steps by taking finite direct products, wreath products with $S_{\omega}$, and closed supergroups.
    \end{fact}

    At the level of structures rather than automorphism groups, finite direct product corresponds to taking the disjoint union of finitely many structures and placing each in its own new unary predicate, while the wreath product with $S_{\omega}$ corresponds to taking an equivalence relation with infinitely many classes and placing in each class a copy of the given structure, and a closed supergroup corresponds to a reduct.

    We next analyze how these operations affect the exponential generating function.
	
\begin{proposition}\label{prod_wr}
	Let $G$ and $H$ be some oligomorphic permutation groups. Then the following hold.
\begin{enumerate}
\item $\eg(\oi_n(G\times H))=\eg(\oi_n(G))\eg(\oi_n(H))$,
\item $\eg(\oi_n(G\wr H))(x)=\eg(\oi_n(H))(\eg(\oi_n(G)(x)-1)$, in particular $\eg(\oi_n(G\wr S_{\omega}))(x)=\exp(\eg(\oi_n(G)-1))$.
\end{enumerate}
\end{proposition}

\begin{proof}
	See for instance~\cite{cameron2000sequences}, Section 4.
\end{proof}

	Proposition~\ref{prod_wr} immediately implies that following.

\begin{corollary}\label{prod_wr_slow}
	Let $G\leq \sym(X)$ and $H\leq \sym(Y)$ be some oligomorphic permutation groups such that the exponential generating functions of the sequences $\oi_n(G)$ and $\oi_n(H)$ both converge at $c$. Then the same holds for the sequences $\oi_n(G\times H)$ and $\oi_n(G\wr S_{\omega})$.
\end{corollary}

    Using this description we obtain the following.
	
\begin{theorem}\label{ms_slow}
	Let $\fa$ be a countable $\omega$-categorical monadically stable structure. Then the power series $\eg(\oi_n(\fa))(x)$ converges everywhere. In particular, for all $c\in \mathbb{R}$ we have $\oi_n(\fa)<n!/c^n$ if $n$ is large enough.
\end{theorem}

\begin{proof}
	Clearly, the statement is true for finite structures. Otherwise, the statement follows from a straightforward induction argument using Fact \ref{mstab induct} and Corollary~\ref{prod_wr_slow}.
\end{proof}

\begin{definition}
	We say that a structure $\fa$ is \emph{cellular} if there exists a $\emptyset$-definable equivalence relation $E$ such that the action of $\aut(\fa)$ on $A/E$ contains a finite direct product of symmetric groups of at most countable sets.
\end{definition}

	Different characterizations of cellular structures are given in~\cite{bodor2024classification}, in particular we know that all cellular structures are monadically stable.
In terms of labelled growth we know the following.
	
\begin{theorem}[~\cite{bodor2024classification}, Theorem 3.54]\label{cd}
	A structure $\fa$ is cellular if and only if there exist $c,d\in \R, d<1$ such that $\oi_n(\fa)<cn^{dn}$ for all $n$.
\end{theorem}

\begin{notation}
    We denote by $B_n$ the $n^{th}$ (first-order) Bell number, i.e., the number of partitions of $[n]$ (see the Sloane sequence~\href{https://oeis.org/A000110}{A000110}).\\
    We denote by $B_n^{(2)}$ the $n^{th}$ second-order Bell number, i.e., the number or pairs of partitions $(P,Q)$ of $[n]$ such that $P$ is a refinement of $Q$ (see the Sloane sequence~\href{https://oeis.org/A000258}{A000258}).
\end{notation}

\begin{observation}\label{bb2}
	The second-order Bell numbers can be calculated by the formula $B_n^{(2)}=\sum_{k=1}^n\bracenom{n}{k}B_n$.
\end{observation}

	Both the first and second-order Bell numbers appear as orbit growth functions, as the following example shows.
	
\begin{example}
	Let $\fa=(\omega;E)$ where $E$ is an equivalence relation with infinitely many infinite classes. Then $\oi_n(\fa)=B_n$ and $\oo_n(\fa)=B_n^{(2)}$.
\end{example}

\begin{remark}\label{cd_bell}
	Note that the upper bound $cn^{dn}$ where $d<1$ is related to the \emph{allied Bell numbers}, that is the number of partitions of $[n]$ where there is a fixed upper bound on the parts, see the discussion in~\cite{bodirsky2021permutation}, Section 6.1. In particular, we know that if $d<1$ then $cn^{dn}<B_n$ if $n$ is large enough.
\end{remark}

	Having given the upper bound for cellular and monadically stable structures, we now continue to the lower bound for monadically stable structures that are not cellular, and afterwards for structures that are not monadically stable.

\begin{lemma}\label{not_cell}
	If $\fa$ is a $\omega$-categorical and monadically stable but not cellular, then $\fa$ has a $\emptyset$-definable equivalence relation on singletons with infinitely many infinite classes.
\end{lemma}

\begin{proof}
	This follows from either Lachlan's description of $\omega$-categorical monadically stable structures in~\cite{lachlan1992}; or from the explicit description of hereditarily cellular groups given in~\cite{bodor2024classification}.
\end{proof}

	Lemma~\ref{not_cell} immediately implies the following.

\begin{corollary}\label{not_cell2}
	If $\fa$ is a $\omega$-categorical and monadically stable but not cellular then $\oi_n(\fa)\geq B_n$.
\end{corollary}

\subsection{Structures that are not monadically stable}

	We start with a few easy observations about the labelled growth profile of $\omega$-categorical structures.
	
\begin{lemma}\label{add_constants_gen}
	Let $G\leq \sym(A)$ be an oligomorphic group for some a countably infinite set $A$, and let $a\in A$. Then $\oi_n(G_a)\leq (n+1)\oi_{n+1}(G)$.
\end{lemma}

\begin{proof}
	If two $n$-tuples containing the element $a$ at the same position are in different $G_a$-orbits then they are also in different $G$-orbits. Therefore, the number of injective $n$-orbits of $G_a$ corresponding to tuples containing $a$ is at most $n\oi_n(G)$. If two $n$-tuples $\bar{b}$ and $\bar{c}$ are disjoint from $a$ and they are in different $G_a$-orbits then the tuples $a\bar{b}$ and $a\bar{c}$ are in different $G$-orbits. Thus, the number of injective $n$-orbits of $G_a$ disjoint from $a$ is at most $\oi_{n+1}(G)$. From these observations we get the following upper bound for the number of injective $n$-orbits of $G_a$:
\[
\oi_n(G_a)\leq n\oi_n(G)+\oi_{n+1}(G)
\]

	Finally, note that since $A$ is infinite we have $\oi_n(G)\leq \oi_{n+1}(G)$, and thus the statement of the lemma follows.
\end{proof}

\begin{corollary}\label{add_constant}
	Let $\fa$ be a countable $\omega$-categorical structure, and let $F\subset A$ be finite. Then for all $n> |A|$ we have $\oi_n(G)\geq \frac{1}{n(n-1)\dots (n-|F|+1)}\oi_{n-|F|}(G_{F})$.
\end{corollary}

\begin{proof}
	Follows by an easy induction using Lemma~\ref{add_constants_gen}.
\end{proof}

\subsubsection{Unstable structures}

\begin{fact}[{\cite[Theorem 1.1]{Macpherson_1987}}] \label{IP rapid}
Let $\fa$ be an $\omega$-categorical structure that is not NIP. Then there is some polynomial $p$ of degree at least two such that $\oi_n(\fa) > 2^{p(n)}$.
\end{fact}

\begin{lemma}\label{unstable_poly}
    Let $M$ be $\omega$-categorical, NIP, and unstable. Then there is a formula $\phi(x,y)$ (on singletons, with parameters) and a sequence of singletons $I = (a_i : i \in \omega)$ such that $M \models \phi(a_i, a_j) \iff i \leq j.$ Thus $M$ has labeled growth rate $\Omega(n!/p(n))$ for some polynomial $p(n)$.
\end{lemma}
\begin{proof}
    The first part follows from \cite[Theorem 5.8]{simon2022linear}, noting that if $M$ is unstable then the partial 1-type $\{x=x\}$ has op-dimension $\geq 1$, and if $M$ is $\omega$-categorical then $\bigvee$-definable relations are definable.

    The growth rate then easily follows by Corollary \ref{add_constant}.
\end{proof}

We remark that one can get the weaker bound of $\Omega(n!/(2+\varepsilon)^n)$ from a more elementary argument, using \cite{simon2021note} to obtain the order-property on singletons. The lower bound then follows from Lemma~\ref{lemma:half graphs}.

\subsubsection{Stable structures that are not monadically stable}

\begin{lemma}\label{stable_not_mod}
	Let $\fa$ be a countable stable $\omega$-categorical structure which is not monadically stable. Then for all $\varepsilon>0$, we have $\oi_n(\fa)>n^{(2-\varepsilon)n}$ for sufficiently large $n$.
\end{lemma}

\begin{proof}
	By Theorem~\ref{coding}, $\thy(\fa)$ admits coding. Thus there is a formula $\phi(x,y,z)$ with parameters and elements $\set{a_i,b_j,c_{i,j}: i,j\in \omega}$ such that $\phi(a_i,b_j,c_{k,l})$ holds if and only if $(i,j)=(k,\ell)$. Let $F$ be the set of parameters in $\phi(x,y,z)$. 

	Now let $n>40/\varepsilon$ be arbitrary and choose some $\frac{\varepsilon}{20}n+1< m<\frac{\varepsilon}{10}n$. Let $\bar{a}=(a_1,\dots,a_m)$, $\bar{b}=(b_1,\dots,b_m)$, and $\bar{c}=(c_{1,1},\dots,c_{m,m})$. Let us now consider the $n$-tuples of the form $\bar{d}^+=\bar{a}\bar{b}\bar{d}$ where $\bar{d}$ is an injective $(n-2m)$-subtuple of $\bar{c}$. By the properties of elements $a_i,b_i,c_{i,j}$ we know that for different choices of $\bar{d}$ the $\bar{d}^+$ will be in different $n$-orbits of $\aut(\fa)_F$. From this we get
	
\begin{align*}
\oi_n(\aut(\fa)_F)&\geq m^2(m^2-1)\dots (m^2-n-2m+1)\geq ((m-1)^2-n)^{n-2m}\\
&>\Bigl(\frac{\varepsilon^2 n^2}{400}-n\Bigr)^{n-2m}>\Bigl(\frac{\varepsilon^2 n^2}{401}\Bigr)^{(1-\varepsilon/5)n}>n^{(2-\varepsilon/2)n}
\end{align*}

	if $n$ is large enough.
	
	Using Corollary~\ref{add_constant}, it follows that if $n$ is large enough, then
	
\[
\oi_n(\fa)\geq \frac{1}{n^{|F|}}(n-|F|)^{(2-\varepsilon/2)(n-|F|)}\geq n^{(2-\varepsilon)n}.\qedhere
\]
\end{proof}

	The following example shows that we cannot get rid of the $\varepsilon$ in the lower bound above.

\begin{example}
	Let $\fa=(\omega^3;E_1,E_2)$ where $(a_1,a_2,a_3)E_i(b_1,b_2,b_3)$ holds if and only if $a_i=b_i$. Then $\fa$ clearly admits coding as witnessed by the formula $\phi(x,y,z)\Leftrightarrow xE_1z\wedge yE_2z$ and $A=B=\{(i,i,i): i\in \omega\}$ and $C=\{(i,j,0): i,j\in \omega\}$. Note that $\fa$ is also the Fra\"{i}ss\'{e} limit of all finite structures with two equivalence relations $E_1$ and $E_2$. Therefore, $\oi_n(\fa)=B_n^2<n^{2n}$. 
	
	We can get an even slower labelled growth if we get rid of the third coordinate in $\fa$. In this case the labelled growth is the Sloane sequence~\href{https://oeis.org/A059849}{A059849}.
\end{example}

	We are now ready to prove our first main result.

    \begin{theorem}\label{thm:main}
    Let $\fa$ be an $\omega$-categorical structure. Then one of the following holds.
    \begin{enumerate}
        \item $\fa$ is cellular and there exist $c, d \in \R, d<1$ such that $\oi_n(\fa) \leq cn^{dn}$.
        \item $\fa$ is monadically stable but not cellular. In this case $\oi_n(\fa) \geq B_n$ where $B_n$ is the $n^{th}$ Bell number, and for sufficiently large $n$, $\oi_n(\fa) \leq n!/c^n$ for all $c > 0$.
        \item $\fa$ is not monadically stable and there is some polynomial $p(n)$ such that $\oi_n(\fa) \geq n!/p(n)$.
    \end{enumerate}
\end{theorem}
\begin{proof}
	The case when $\fa$ is cellular is handled by Theorem~\ref{cd}. When $\fa$ is monadically stable but not cellular the lower bound is given by Corollary~\ref{not_cell2} and the upper bound is given by Theorem~\ref{ms_slow}. The lower bound in the unstable case is given by Fact \ref{IP rapid} and Lemma~\ref{unstable_poly}. Finally, if $\fa$ is stable but not monadically stable the lower bound follows from Lemma~\ref{stable_not_mod} since $n^n\geq n!$.
\end{proof}

\section{Shrub-depth and the labelled growth rate of hereditary graph classes}

The orbit-counting in the previous section is equivalent to counting structures in hereditary classes that arise from taking the age of an $\omega$-categorical relational structure with quantifier elimination, as discussed in Section \ref{sec:prelim}. In this section, we prove Theorem \ref{thm:intrographgap}, showing all graph classes with sufficiently slow growth must arise from $\omega$-categorical monadically stable structures. (We expect that the same should be true for relational structures, but the crucial results of \cite{mahlmann2025forbidden} are proved only for graphs.)

We refer to \cite[Section 1]{ossona2025transducing} for the fact that the following definition is equivalent to more standard ones.

\begin{definition}
    Let $\CC$ be a hereditary class of graphs. Then $\CC$ has \emph{bounded shrub-depth} if $\CC$ can be defined on singletons in a class of colored forests of height $k$ for some $k \in \omega$. 
    
    More formally, there is a 1-dimensional FO-interpretation in the sense of \cite[Section 2.4]{mahlmann2025forbidden} (which does not allow for quotients by definable equivalence relations) of $\CC$ in some class of forests of height at most $k \in \omega$ equipped with additional unary predicates.
\end{definition}

\begin{notation}
    Let $\CC$ be a class of structures. Given $n \in \omega$, we let \emph{$\CC_{[n^-]}$} denote the class of all expansions of $\CC$ by $n$ new constant symbols.
\end{notation}

\begin{theorem}[Special case of {\cite[Theorem 1]{Burris1994model}}]
Let $\CC$ be a hereditary class of finite relational structures. The following are equivalent.
\begin{enumerate}
    \item $\CC$ is a finite union of ages of $\omega$-categorical structures.
    \item For every $n \in \omega$, there does not exist $\set{M_i : i \in \omega} \subset \CC_{[n^-]}$ such that for every $j \neq k \in \omega$, there is no $N \in \CC_{[n^-]}$ that both $M_j$ and $M_k$ embed into.
\end{enumerate}
\end{theorem}

\begin{corollary}[cf.~{\cite[Theorem 4.16]{pouzet2024well}}] \label{cor:fin union}
    Let $\CC$ be a hereditary class of finite relational structures such that $\CC_{[n^-]}$ is well-quasi-ordered under embeddability for every $n \in \omega$. Then $\CC$ is a finite union of ages of $\omega$-categorical structures.
\end{corollary}

\begin{proposition} \label{prop:sd catms}
    Let $\CC$ be a hereditary class of graphs. Then $\CC$ has bounded shrub-depth if and only if $\CC$ is a finite union of ages of $\omega$-categorical monadically stable graphs.
\end{proposition}

\begin{proof}
    $(\Rightarrow)$ Let $\CC$ have bounded shrub-depth. It follows from \cite[Corollary 3.9]{ganian2019shrub} that $\CC_{[n^-]}$ is well-quasi-ordered under embeddability for every $n \in \omega$. So by Corollary \ref{cor:fin union}, $\CC$ is the finite union of ages of $\omega$-categorical graphs $\set{G_i : i \in [n]}$. Also, $\CC$ is monadically stable (see, e.g., \cite[Figure 1]{nesetril2021classes}). By \cite[Proposition 2.5]{braunfeld2022existential}, monadic stability is a property of the universal fragment of a theory, and so each $G_i$ is monadically stable as well.

    $(\Leftarrow)$ As a finite union of classes of bounded shrub-depth still has bounded shrub-depth, let $\CC$ be the age of an $\omega$-categorical monadically stable structure $M$. If $M$ is finite, then $\CC$ has bounded shrub-depth. Since the inductive steps for building up an $\omega$-categorical monadically stable structure from finite structures described in and after Fact \ref{mstab induct} all preserve bounded shrub-depth, the result follows.
\end{proof}

For the next two definitions, we refer to \cite[Figure 2]{mahlmann2025forbidden} for an illustration.

\begin{definition}
    Let $G$ and $H$ be graphs on the same vertex set, and let $\PP$ be a partition of the vertex set. Then \emph{$G$ is a $\PP$-flip of $H$} if $G$ can be obtained from $H$ by complementing the edge relation on some subset of (not necessarily distinct) pairs of parts of $\PP$. More formally, there is some symmetric $\II \subset \PP^2$ such that, letting $\PP_x$ be the part of $\PP$ containing the vertex $x$, we have $H \models E(x,y)$ if and only if either ($G \models E(x,y)$, and $(\PP_x, \PP_y) \not \in \II$) or ($G \models \neg E(x,y)$, and $(\PP_x, \PP_y) \in \II$).
\end{definition}

\begin{definition}
     We let $P_t$ denote a path on $t$ vertices, and $kP_t$ the disjoint union of $k$ copies of $P_t$. Identify the vertex set of $kP_t$ with $[k] \times [t]$, where $(i,j)$ is the $j^{th}$ vertex on the $i^{th}$ path, and let $\PP$ be the partition of the vertex set where $[k] \times \set{j}$ is a part for each $j \in [t]$. Then a \emph{flipped $kP_t$} is a $\PP$-flip of $kP_t$.

     We let $H_t$ denote the half-graph of order $t$.
\end{definition}
   
\begin{definition}
    A bipartite graph $G = (U, V; E)$ is \emph{semi-induced} in a graph $H = (W; F)$ if there is a function $f \colon U \cup V \to W$ such that $f(E) = F \upharpoonright (f(U) \times f(V))$.
\end{definition}

\begin{lemma}\label{lemma:half graphs}
    Let $\CC$ be a hereditary class of graphs containing arbitrarily large semi-induced half-graphs. Then for any $\varepsilon>0$ we have $\oi_n(\CC)>\frac{n!}{(2+\varepsilon)^n}$ for sufficiently large $n$.
\end{lemma}

\begin{proof}
	First, suppose $n=2k$ is even. Let $H_k \in \CC$ be a semi-induced half-graph of order $k$, with  $a_1,\dots,a_k,b_1,\dots,b_k \in H_k$ such that $\phi(a_i,b_j)\Leftrightarrow i<j$. Then for every pair of permutations $\sigma, \tau \in \sym_k$, we may consider the labeled copy of $H_k$ labelling each $a_i$ by $\sigma(i)$ and each $b_j$ by $\tau(j)$. All of these labellings are distinct, and so  $\oi_{n}(\CC)\geq (k!)^2\geq \frac{(2k)!}{2^{2k}} = \frac{n!}{2^{n}}$ for sufficiently large $n$ (using Stirling's approximation for the last inequality).
    
    If $n$ is odd with $n=2k+1$, we have
\[\oi_n(\CC)\geq \oi_{2k}(\CC)\geq \frac{(n-1)!}{2^{n-1}}=\frac{2}{n}\frac{n!}{2^{n}}\geq \frac{n!}{(2+\varepsilon)^n}\] for $m$ sufficiently large.
\end{proof}

\begin{theorem} [{\cite[Theorem 1.1]{mahlmann2025forbidden}}]
    Let $\CC$ be a hereditary graph class of unbounded shrub-depth. Then $\CC$ contains either a flipped $3P_t$ for arbitrarily large $t$ or a semi-induced $H_t$ for arbitrarily large $t$.
\end{theorem}

\begin{lemma} \label{lemma:sd lower}
    Let $\CC$ be a hereditary graph class of unbounded shrub-depth. Then $\oi_n(\CC) > \frac{n!}{10^n}$ for sufficiently large $n$.
\end{lemma}
\begin{proof}
Case 1: Suppose $\CC$ contains arbitrarily large semi-induced $H_n$s. In this case we are done by Lemma \ref{lemma:half graphs}.

Case 2: Suppose $\CC$ contains arbitrarily large flipped $3P_k$s. We initially suppose $n = 3k$ with $k \in \omega$.  We will show that after expanding a flipped $3P_k$ by three unary predicates (one for each path), we may definably recover $3P_k$. But first, note that number of distinct labellings of $P_k$ is at least $k!/2$, since given $P_k = \set{p_1, \dots, p_k}$ and $\sigma, \tau \in \sym_k$ such that $\sigma$ is not the reverse of $\tau$ when written in one-line notation, the labelling that labels each $p_i$ by $\sigma(i)$ is distinct from the one that labels each $p(i)$ by $\tau(i)$. Thus the number of labellings of $3P_k$ is at least $(\frac{k!}{2})^3 \geq \frac{(3k)!}{27^k}$ (using Stirling's approximation). For each vertex of a flipped $3P_k$ we have 3 different choices as to which of the 3 unary predicates it belongs. Thus, paying a factor of $3^{3k}=27^k$ for the expansion by the three unary predicates, the number of labellings of a flipped $3P_k$ would be at least $\frac{(3k)!}{729^k} = \frac{(3k)!}{9^{3k}} = \frac{n!}{9^{n}}$. As at the end of the proof of Lemma \ref{lemma:half graphs}, we obtain the desired lower bound for $\CC_n$ for $n$ not necessarily divisible by 3.

So we now show that there is a formula $\phi(x,y)$ in the language of graphs with three unary predicates such that on any flipped $3P_k$ expanded by one unary predicate for each path (in the unflipped graph), $\phi(x,y)$ defines the edge relation of $3P_k$. We essentially just repeat three times the argument of \cite[Lemma 4.11]{mahlmann2025forbidden} that definably recovers a single path, but we recall the argument here.  Let $H$ be a flipped $3P_k$. By \cite[Lemma 2.8]{mahlmann2025forbidden}, there is a unique minimum-cardinality partition $\PP$ of the vertices of $3P_k$ so that $H$ can be obtained as a $\PP$-flip. Let $C_i$ for $i \in \set{0,1,2}$ be the three unary predicates expanding $H$, let $E$ be the edge relation, and take all addition to be mod 2. For each $i \in \set{0,1,2}$, let $\pi_i(x,y)$ state that $x \in C_i$, $y \in C_{i+1}$ and $x, y$ have the same neighborhood in $C_{i+2}$. Then $\pi_i(x,y)$ holds if and only if $x$ and $y$ are in the appropriate unary predicates and in the same $\PP$-class. Next, let $\psi_i(x,y) := x, y \in C_i \wedge\exists z(\pi_i(y,z) \wedge E(x,z))$. Then $\psi_i(x,y)$ holds if and only if $x,y$ are in $C_i$ and their $\PP$-classes have been flipped with respect to each other in the (unique) $\PP$-flip that yields $H$ from $3P_k$. Thus we let $\phi(x,y) := \bigvee_{i \in \set{0,1,2}} x \neq y \wedge x, y \in C_i \wedge (\psi_i(x,y) \leftrightarrow \neg E(x,y))$.
\end{proof}

\begin{theorem} \label{thm:graph gap}
    Let $\CC$ be a hereditary graph class. Then one of the following holds.
    \begin{enumerate}
        \item $\CC$ has bounded shrub-depth and $\oi_n(\CC) < \frac{n!}{c^n}$ for every $c \in \R$, for sufficiently large $n$.
        \item $\CC$ has unbounded shrub-depth, and $\oi_n(\CC) > \frac{n!}{10^n}$, for sufficiently large $n$.
    \end{enumerate}
\end{theorem}
\begin{proof}
    If $\CC$ has unbounded shrub-depth, the lower bound is given by Lemma \ref{lemma:sd lower}. If $\CC$ has bounded shrub-depth, then by Proposition  \ref{prop:sd catms} it is a finite union of ages of $\omega$-categorical monadically stable graphs $\set{G_i : i \in [k]}$. Then $\oi_n(\CC) \leq \Sigma_{i \in [k]} \oi_n(G_i)$, and the upper bound follows from Theorem \ref{ms_slow}.
\end{proof}

\section{Growth rates in weakly finite sets via absoluteness}
\label{sec:growth-rates-in-weakly-finite-sets}

Since we will be discussing a number of properties of sets similar to those in Fact~\ref{fact:WalczakTypkeCorrespondence}, we will introduce some systematic terminology for converting model-theoretic properties to set-theoretic properties as suggested by the correspondence in Fact~\ref{fact:WalczakTypkeCorrespondence}. For the sake of the following definition, we will fix the bookkeeping convention that whenever we write $L[T]$ for a theory $T$, we have coded the language and formulas of $T$ as ordinals. In particular, we will only ever write $L[T]$ in a context in which the language of $T$ has an implicitly given well-ordering. The significance of this is that it means that for any theory $T$, $T \in L[T]$ and $L[T]$ satisfies \ZFC.

\begin{definition}\label{defn:nebulously}
  Given a property of theories $P$, a set $A$ is \emph{nebulously $P$} if for any countable theory $T$ admitted by $A$, $L[T]$ satisfies `$T$ is a $P$ theory.'
\end{definition}

The word `nebulous' is meant to evoke a connection to amorphicity, but we have avoided the more obvious terminology `amorphously $P$' given that for most of the properties $P$ considered, the nebulous version is strictly weaker than amorphicity.

So, rephrasing Fact~\ref{fact:WalczakTypkeCorrespondence} in this terminology, we have, for instance, that a set $A$ is strictly Mostowski-finite if and only if it is nebulously $\omega$-categorical and NSOP. Note that as the word is defined here, the class of `nebulously o-minimal theories' does not really make sense, in that o-minimality is defined in terms of a designated order relation (although nebulous convex orderability and nebulous distality do make sense, of course). It would be relatively easy to extend the definition to this kind of property, in order to include notions like o-amorphicity, but in the interest of reducing technicality we have not done this.

We should stop for a moment to comment on the use of the inner model $L[T]$ in Definition~\ref{defn:nebulously}. This is a technical device to avoid needing to worry about finding the correct choice-free analog of a given model-theoretic property if it is not phrased in a set-theoretically absolute way. Working in $L[T]$ guarantees that any \ZFC-provable arithmetical properties\footnote{Or more generally $\Sigma^1_2$ properties by Shoenfield absoluteness.} of $T$ will yield \ZF-provable properties of $A$. Note though that most model-theoretic properties of theories are set-theoretically absolute in the context of countable theories. For example, uncountable categoricity is not prima facie an arithmetical property of countable theories, but was shown to be by Andrews and Makuluni in \cite{andrews2013}.

Recall that a \emph{one-dimensional interpretation} is an interpretation of a theory $T_0$ into a theory $T_1$ using a definable quotient of a definable set in the home sort.

\begin{proposition}\label{prop:nebulous-free-omega-cat}
  $(\ZF)$ If $P$ is a property of first-order theories such that \ZFC\ proves `$P$ is preserved under one-dimensional interpretation\footnote{In the sense that if $T$ satisfies $P$ and $T'$ has a one-dimensional interpretation in $T$, then $T'$ satisfies $P$.} and $\thy(\omega,+,\times)$ is not $P$', then any nebulously $P$ set is power-Dedekind-finite.
\end{proposition}
\begin{proof}
  Assume that $A$ is not power-Dedekind-finite and let $f: A \to \omega$ be a surjection. Let $\Lc$ be a language with a binary relation $E$ and two ternary relations $F_+$ and $F_\times$. Make $A$ into an $\Lc$-structure such that $E(a,b)$ holds if and only if $f(a) = f(b)$ and $F_{\bullet}(a,b,c)$ holds if and only if $f(a)\bullet f(b)=f(c)$ for $\bullet \in \{{+},{\times}\}$. Let $T = \thy(A)$. Clearly we have that $T$ has a one-dimensional interpretation of $\thy(\omega,+,\times)$ (and moreover this interpretation clearly lives in $L[T]$). Therefore, since $P$ is preserved under one-dimensional interpretations, we have that $T$ is not $P$, whereby $A$ is not nebulously $P$.
\end{proof}

As a corollary we get (in \ZF) that for many of the tameness properties $P$ considered by model theorists (e.g., stability, NIP, higher-arity generalizations thereof, NTP$_2$, NSOP, NATP, rosiness), any nebulously $P$ set is power-Dedekind-finite (i.e., nebulously $\omega$-categorical). Moreover, we can mildly strengthen Fact~\ref{fact:WalczakTypkeCorrespondence}.

\begin{corollary}\label{cor:Walczak-Typke-rephrase}
  $(\ZF)$ Fix a set $A$.
  \begin{itemize}
  \item $A$ is power-Dedekind-finite if and only if it is nebulously $\omega$-categorical.
  \item $A$ is strictly Mostowski-finite if and only if it is nebulously NSOP.
  \item $A$ is amorphous if and only if it is nebulously strongly minimal.
  \end{itemize}
\end{corollary}
\begin{proof}
  This follows immediately from Fact~\ref{fact:WalczakTypkeCorrespondence} and Proposition~\ref{prop:nebulous-free-omega-cat}.
\end{proof}

Using the above machinery, we can now prove purely finitary combinatorial properties of certain kinds of weakly finite sets in \ZF{} using model-theoretic results originally shown in \ZFC.

\begin{definition}
  Fix a set $A$ and a family $\Cfr = (C_n)_{n \in \omega}$ of sets with $C_n \subseteq \Pc(A^n)$ for each $n \in \omega$.   We say that such a family $\Cfr$ is \emph{countable} if $\bigsqcup_{n \in \omega} C_n$ is countable. We say that $\Cfr$ is \emph{well-orderable} if $\bigsqcup_{n \in \omega} C_n$ is well-orderable.

  We let $s_n(\Cfr)$ be the number of atoms in the Boolean subalgebra of $\Pc(A^n)$ generated by $C_n$ if $C_n$ is finite and we let $s_n(\Cfr) = \infty$ otherwise.

  If each $C_n$ is a subset of $\Pc(A^{(n)})$, then we will also write $\ell_n(\Cfr)$ for the number of atoms in the Boolean subalgebra of $\Pc(A^{(n)})$ generated by $C_n$ if $C_n$ is finite (and again we let $\ell_n(\Cfr) = \infty$ if $C_n$ is infinite).
\end{definition}

Note that by an easy argument $\bigsqcup_{n \in \omega} C_n$ is well-orderable (resp.~countable) if and only if $\bigcup_{n \in \omega} C_n$ is well-orderable (resp.~countable). Note also that if $C_n$ is finite, then it is immediate that $|C_n| \leq 2^{s_n(\Cfr)}$ (and $|C_n| \leq 2^{\ell_n(\Cfr)}$ if $\ell_n(\Cfr)$ is defined).

\begin{definition}
  Given a set $A$ and a well-orderable family $\Cfr = (C_n)_{n \in \omega}$ with $C_n \subseteq \Pc(A^n)$ for each $n \in \omega$, the \emph{language induced by $\Cfr$} is a language with an $n$-ary predicate symbol $Q$ for each element of $C_n$ (with some coding as ordinals chosen using a given well-ordering of $\bigsqcup_{n\in \omega} C_n$). The \emph{structure on $A$ induced by $\Cfr$} is the obvious interpretation of $A$ as a structure in the language induced by $\Cfr$. We write $\thy(\Cfr)$ for the theory of structure.
\end{definition}

\begin{lemma}\label{lem:well-orderable-implies-countable}
  $(\ZF)$ If $A$ is power-Dedekind-finite and $\Cfr= (C_n)_{n \in \omega}$ with $C_n \subseteq \Pc(A^n)$ is well-orderable, then it is countable and each $C_n$ is finite.
\end{lemma}
\begin{proof}
  Since $\bigsqcup_{n \in \omega} C_n$ is well-orderable, it is immediate that if each $C_n$ is finite, then $\bigsqcup_{n \in \omega} C_n$ is countable. By Fact~\ref{fact:WalczakTypkeCorrespondence}, $\thy(\Cfr)$ is $\omega$-categorical and so only has finitely many formulas with $n$ free variables up to logical equivalence. Therefore, by extensionality, each $C_n$ is finite.
\end{proof}

We have the following general result, which we will use in a few specific cases. Given a first-order theory $T$, let $S^\ast_n(T)$ denote the space of $n$-types of $n$-tuples of distinct elements.

\begin{proposition}\label{prop:main-meta-prop}
  Let $P$ be a property of first-order theories and let $\Oc$ be an arithmetically definable set of functions from $\omega$ to $\omega$ that is downwards closed (i.e., if $g \in \Oc$ and $f(n)\leq g(n)$ for all $n \in \omega$, then $f \in \Oc$). Assume that the following is provable in \ZFC:
  \begin{itemize}
    \item[$(\ast)$] If $T$ is a countable $P$ theory, then it is $\omega$-categorical and $(n \mapsto |S_n(T)|) \in \Oc$.
  \end{itemize}
  Then \ZF{} proves that for any set $A$ and well-orderable $\Cfr = (C_n)_{n \in \omega}$ (with $C_n \subseteq \Pc(A^n)$ for each $n \in \omega$), if $\thy(\Cfr)$ satisfies $P$ in $L[\thy(\Cfr)]$, then $C_n$ is finite for each $n$ and $(n \mapsto s_n(\Cfr)) \in \Oc$.

  The same statement also holds if we replace $S_n(T)$ with $S^\ast_n(T)$ and $s_n(\Cfr)$ with $\ell_n(\Cfr)$.
\end{proposition}
\begin{proof}
  By Lemma~\ref{lem:well-orderable-implies-countable}, we have that $\Cfr$ is countable. Let $T = \thy(\Cfr)$.
  Since $T$ is a set of ordinals, $L[T]$ is a model of $\ZFC$.

  Since $T$ satisfies $P$ in $L[T]$, we have by $(\ast)$ that $L[T]$ satisfies that $T$ is $\omega$-categorical and has that $(n \mapsto |S_n(T)|) \in \Oc$. Since $\Oc$ is downwards closed, we have that (again in $L[T]$) that $(n \mapsto s_n(\Cfr)) \in \Oc$. Since $\Oc$ is arithmetically definable, this implies that $(n \mapsto s_n(\Cfr)) \in \Oc$ in our original model $V$ of \ZF{} as well, so we are done.

Note that $s_n(\Cfr)$ is now upper bounded by the number of quantifier-free $n$-types in the structure on $A$ induced by $\Cfr$ and so in particular satisfies that $s_n(\Cfr) \leq |S_n(T)|$ for each $n \in \omega$.

  The proof of the version of the proposition involving $S^\ast_n(T)$ and $\ell_n(\Cfr)$ is the same.
\end{proof}

The easiest application of Proposition~\ref{prop:main-meta-prop} is in the case of amorphous sets, using the following result of Steitz.

\begin{fact}[{\cite{Steitz1992}}]\label{fact:Steitz}
  $(\ZFC)$ If $T$ is a countable $\omega$-stable, $\omega$-categorical theory, then there is a $k \in \omega$ such that $|S_n(T)| \leq 2^{kn^2}$ for all $n \in \omega$.
\end{fact}

\begin{proposition}\label{prop:amorphous-case}
  $(\ZF)$ For any amorphous set $A$ and family $(C_n)_{n \in \omega}$ with $C_n \subseteq \Pc(A^n)$ for each $n$, if $\bigcup_{n \in \omega}C_n$ is well-orderable, then there is a $k$ such that $|C_n| \leq 2^{2^{kn^2}}$ for all $n \in \omega$.
\end{proposition}
\begin{proof}
  The set of functions $f : \omega \to \omega$ satisfying that for some $k \in \omega$, $f(n) \leq 2^{kn^2}$ for all $n \in \omega$ is clearly arithmetically definable. By Fact~\ref{fact:WalczakTypkeCorrespondence}, $A$ is nebulously $\omega$-categorical and $\omega$-stable. Taking `$\omega$-categorical and $\omega$-stable' as our property $P$, we can apply Proposition~\ref{prop:main-meta-prop} and Fact~\ref{fact:Steitz} to get that for any well-orderable family $\Cfr = (C_n)_{n \in \omega}$ as in the statement of the proposition, there is an $m \in \omega$ such that $s_n(\Cfr) \leq 2^{kn^2}$ for all $n \in \omega$. Applying the obvious bound on $|C_n|$ in terms of $s_n(\Cfr)$ gives that $|C_n| \leq 2^{2^{kn^2}}$ for all $n \in \omega$, as required.
\end{proof}

Proposition~\ref{prop:amorphous-case} also applies to MT-ranked sets (such as finite unions of amorphous sets), since MT-rankedness is equivalent to nebulous $\omega$-stability. In Section~\ref{sec:structure-theorem-amorphous}, we will give a much stronger structure theorem for amorphous sets, strengthening some of the results of Truss in \cite{Truss_1995}.

\subsection{Nebulous (monadic) stability}
\label{sec:nebulous-monadic-stability}

\begin{proposition}\label{prop:monadic-stability-char} 
  $(\ZFC)$ $T$ is monadically stable if and only if it is NSOP and does not admit tuple-coding.
\end{proposition}
\begin{proof}
  If $T$ is monadically stable, then it is stable and therefore NSOP and moreover is monadically NIP and therefore does not admit tuple-coding by \cite[Thm.~1.1]{braunfeld2021characterizations}. The other direction follows by \cite[Thm.~1.1]{braunfeld2021characterizations} and Theorem~\ref{coding}.
\end{proof}

\begin{lemma}\label{lem:coding-finitary}
  $(\ZFC)$ A first-order theory $T$ admits tuple-coding if and only if there is a formula $\varphi(\xbar,\ybar,z)$ such that for every $m$, there are sequences $(\abar_i)_{i < m}$, $(\bbar_j)_{j < m}$, and $(c_{i,j})_{i,j < m}$ such that for all $i,j,k,\ell < m$, $\varphi(\abar_i,\bbar_i,c_{k,\ell})$ holds if and only if $i=k$ and $j = \ell$. An analogous statement is true for admitting coding.
\end{lemma}
\begin{proof}
  This is immediate by compactness.
\end{proof}

Recall that a binary relation $R$ on a set $A$ is \emph{well-founded} if there is a function $f : A \to \mathrm{Ord}$ such that for any $a,b \in A$, if $a \mathrel{R} b$, then $f(a) > f(b)$. $R$ is \emph{ill-founded} if it is not well-founded. Without dependent choice, it is not always true that a binary relation $R$ is ill-founded if and only if there is a sequence $(a_i)_{i \in \omega}$ satisfying $a_i \mathrel{R} a_{i+1}$ for all $i$, but we will use the following closely related fact.

\begin{fact}\label{fact:ill-founded-forcing}
  $(\ZF)$ A binary relation $R$ on a set $A$ is \emph{ill-founded} if and only if there is a forcing extension $V[G]$ of the universe containing a sequence $(a_i)_{i \in \omega}$ of elements of $A$ such that $a_i \mathrel{R} a_{i+1}$ for all $i \in \omega$.
\end{fact}

We will use the following lemma several times.

\begin{lemma}\label{lem:omega-cat-ill-founded}
  $(\ZF)$ If $T$ is $\omega$-categorical and $(\abar_i)_{i \in \omega}$ is some sequence of finite tuples (of possibly different length) in the unique countable model of $T$, then for any $M \models T$, the set of tuples $(\bbar_0\dots,\bbar_{n-1}) \in (M^{<\omega})^{<\omega}$ satisfying $(\bbar_0,\dots,\bbar_{n-1}) \equiv (\abar_0,\dots,\abar_{n-1})$ is ill-founded under extension.
\end{lemma}
\begin{proof}
  For each $n$, the type of $(\abar_0,\dots,\abar_{n-1})$ is isolated by a formula and $T$ moreover says that there exists a $\cbar$ such that $(\abar_0,\dots,\abar_{n-1},\cbar) \equiv (\abar_0,\dots,\abar_{n-1},\abar_n)$. Therefore this extension property holds in any model of $T$ and we have that the tree of tuples in the statement of the lemma is ill-founded in any $M \models T$.
\end{proof}

\begin{proposition}\label{prop:nebulously-not-tuple-coding}
  $(\ZF)$ For any set $A$, the following are equivalent.
  \begin{enumerate}
  \item\label{mon-NIP} $A$ is nebulously monadically NIP.
  \item\label{neb-tup-code} $A$ nebulously does not admit tuple-coding.
  \item\label{finitary-tup-code-1} For any $k \in \omega$ and $D \subseteq A^{2k+1}$, there is an $m$ such that for any $X,Y \subseteq A^k$ and $Z \subseteq A$, if $|X|,|Y| > m$, then $D \cap (X \times Y \times Z)$ is not the graph of a bijection $f : X \times Y \to Z$.
  \item\label{finitary-tup-code-2} For any $k \in \omega$ and $D \subseteq A^{2k+1}$, the binary relation $(X,Y,Z) < (X',Y',Z')$ defined by $X \subset X' \wedge Y \subset Y' \wedge Z \subset Z'$ on the set \( \{(X,Y,Z) \in \Pc(A^k)\times \Pc(A^k) \times \Pc(A) : D \cap (X \times Y\times Z)~\text{is the graph of a bijection}~f:X\times Y \to Z\} \) is well-founded.
  \end{enumerate}
\end{proposition}
\begin{proof}
  The fact that (\ref{mon-NIP}) and (\ref{neb-tup-code}) are equivalent follows from \cite[Thm.~1.1]{braunfeld2021characterizations}.

  To see that (\ref{neb-tup-code}) implies (\ref{finitary-tup-code-1}), assume that $A$ nebulously does not admit tuple-coding. Fix $D \subseteq A^{2k+1}$ and let $\Afr$ be the structure with underlying set $A$ and a single $(2k+1)$-ary relation $Q$ interpreted as $D$. Let $T = \thy(\Afr)$. Since $T$ does not admit tuple-coding (in $L[T]$), we have by Lemma~\ref{lem:coding-finitary} that there is an $m$ such that there does not exist $(\abar_i)_{i < m}$, $(\bbar_j)_{j < m}$, and $(c_{ij})_{i,j < m}$ in any model of $T$ such that $Q(\abar_i,\bbar_j,c_{j,\ell})$ holds if and only if $i=k$ and $j = \ell$. This is expressible as a first-order sentence, so it must hold in $\Afr$ as well. It is immediate that this is equivalent to (\ref{finitary-tup-code-1}).

  It's clear that (\ref{finitary-tup-code-1}) implies (\ref{finitary-tup-code-2}), so we just need to show that (\ref{finitary-tup-code-2}) implies (\ref{neb-tup-code}). We will prove the contrapositive. Assume that (\ref{neb-tup-code}) fails. Fix a structure $\Afr$ with countable language witnessing this. Let $T = \thy(\Afr)$. In $L[T]$, we can find a formula $\varphi(\xbar,\ybar,z)$ and infinite $X,Y \subseteq M^k$ and $Z \subseteq M$ witnessing that $T$ admits tuple-coding. Call an $(k\cdot \ell+k\cdot \ell+\ell^2)$-type $p((\xbar_i)_{i < \ell},(\ybar_j)_{j < \ell},(z_{i,j})_{i,j< \ell})$ \emph{good} if for any $(\abar_i)_{i < \ell}$, $(\bbar_j)_{j < \ell}$, $(c_{i,j})_{i,j<\ell}$ realizing $p$,
  \begin{itemize}
  \item $\varphi(\abar_i,\bbar_j,c_{i',j'})$ holds if and only if $i = i'$ and $j = j'$ and
  \item there is an extension of $(\abar_i)_{i < \ell}$, $(\bbar_j)_{j < \ell}$, $(c_{i,j})_{i,j<\ell}$ to an infinite witness of tuple-coding.
  \end{itemize}
  Given good types $p((\xbar_i)_{i < \ell},(\ybar_j)_{j < \ell},(z_{i,j})_{i,j< \ell})$ and $q((\xbar_i)_{i < \ell'},(\ybar_j)_{j < \ell'},(z_{i,j})_{i,j< \ell'})$ with $\ell < \ell'$, say that $q$ is a \emph{good extension} of $p$ if the restriction of $q$ to the variables $(\xbar_i)_{i < \ell},(\ybar_j)_{j < \ell},(z_{i,j})_{i,j< \ell}$ is $p$. Clearly we have that any good type admits a good extension and so the collection of good types strictly ordered under good extension is ill-founded in $L[T]$. By absoluteness of well-foundedness, this implies that the same holds in $V$ as well. Since $T$ is $\omega$-categorical, we have by Lemma~\ref{lem:omega-cat-ill-founded} that the collection of triples $(X,Y,Z)$ of finite sets with $X \subseteq A^k$, $Y \subseteq A^k$, and $Z \subseteq A$ such that the type of $(X,Y,Z)$ is good witnesses that the binary relation in (\ref{finitary-tup-code-2}) is ill-founded.
\end{proof}

By essentially the same proof, we get the following.

\begin{proposition}\label{prop:nebulously-not-coding}
  $(\ZF)$ For any set $A$, the following are equivalent.
  \begin{enumerate}
  \item\label{neb-code} $A$ nebulously does not admit coding.
  \item\label{finitary-code-1} For any $D \subseteq A^{3}$, there is an $m$ such that for any $X,Y,Z \subseteq A$, if $|X|,|Y| > m$, then $D \cap (X \times Y \times Z)$ is not the graph of a bijection $f : X \times Y \to Z$.
  \item\label{finitary-code-2} For any $D \subseteq A^{3}$, the binary relation $(X,Y,Z) < (X',Y',Z')$ defined by $X \subset X' \wedge Y \subset Y' \wedge Z \subset Z'$ on the set \( \{(X,Y,Z) \in \Pc(A)^3 : D \cap (X \times Y\times Z)~\text{is the graph of a bijection}~f:X\times Y \to Z\} \) is well-founded. \qed
  \end{enumerate}
\end{proposition}

We can give a particularly simple characterization of nebulous stability using the following result of Simon.

\begin{fact}[{\cite{Simon2021}}]\label{fact:Simon}
  A theory $T$ is stable if and only if every formula $\varphi(x,y)$ (with parameters) in two variables is stable.
\end{fact}

\begin{proposition}\label{prop:nebulous-stability-char}
  $(\ZF)$ For any set $A$, the following are equivalent.
  \begin{enumerate}
  \item\label{not-neb-stab} $A$ is not nebulously stable.
  \item\label{not-neb-stab-every-n} There is a $D \subseteq A^2$ such that for any $n \in \omega$, there are $(a_i)_{i < n}$ and $(b_j)_{j < n}$ in $A$ such that $(a_i,b_j) \in D$ if and only if $i < j$.
  \item\label{not-neb-stab-ill-founded} There is a $D \subseteq A^2$ such that the set of sequences $(a_0,b_0,\dots,a_{n-1},b_{n-1}) \in (A^2)^{<\omega}$ satisfying that $(a_i,b_j) \in D$ if and only if $i < j$ is ill-founded under extension.
  \item\label{not-neb-stab-ill-founded-Q} There is a $D \subseteq A^2$ such that the set of pairs of finite partial functions $f,g : {\subseteq}\mathbb{Q} \to A^2$ with the same domain satisfying that for all $i,j \in \mathrm{dom}(f)$, $(f(i),g(j))$ is ill-founded under extension.
  \end{enumerate}
\end{proposition}
\begin{proof}
  To show that (\ref{not-neb-stab}) implies (\ref{not-neb-stab-ill-founded}), 
  first assume that $A$ is not power-Dedekind-finite. Fix a surjection $f : A \to \omega$ and let $D = \{(x,y) \in A^2 : f(x) < f(y)\}$. This clearly gives (\ref{not-neb-stab-ill-founded}). Now assume that $A$ is power-Dedekind-finite but not nebulously stable. Let $T$ be countable unstable theory admitted by $A$, and let $\Afr$ be a model of $T$ with underlying set $A$. By Fact~\ref{fact:WalczakTypkeCorrespondence}, $T$ is $\omega$-categorical. We can now apply Fact~\ref{fact:Simon} inside $L[T]$ to find a model $\Bfr \models T$ and an unstable $\Bfr$-formula $\varphi(x,y,\cbar)$. Let $(a_i,b_i)_{i \in \mathbb{Q}}$ be an instance of the order property for $\varphi(x,y,\cbar)$. Let $(X_i)_{i \in \omega}$ be a sequence of finite subsets of $\mathbb{Q}$ satisfying that $\bigcup_{i \in \omega} X_i = \mathbb{Q}$. We can now apply Lemma~\ref{lem:omega-cat-ill-founded} to the sequence $\cbar,(a_i,b_i)_{i\in X_0},(a_i,b_i)_{i \in X_1},(a_i,b_i)_{i \in X_2},\dots$ to get the required ill-foundedness in $A$ using the set $\{(x,y) \in A^2 : \Afr \models \varphi(x,y,\dbar)\}$ for some $\dbar \in A$ with $\dbar \equiv \cbar$. Therefore (\ref{not-neb-stab}) implies (\ref{not-neb-stab-ill-founded-Q}).

(\ref{not-neb-stab-ill-founded-Q}) clearly implies (\ref{not-neb-stab-ill-founded}), and (\ref{not-neb-stab-ill-founded}) clearly implies (\ref{not-neb-stab-every-n}). Finally, (\ref{not-neb-stab-every-n}) implies that $A$ supports an unstable theory and so is not nebulously stable.
\end{proof}

\begin{proposition}\label{prop:neb-mon-stab-rad-conv}
  $(\ZF)$ If $A$ is nebulously monadically stable, then for any well-orderable family $\Cfr = (C_n)_{n \in \omega}$ with $C_n \subseteq \Pc(A^{(n)})$ for each  $n$, $C_n$ is finite and for every real $r > 0$, $|C_n|$ is eventually dominated by $2^{n!r^n}$.
\end{proposition}
\begin{proof}
  Since $A$ is nebulously monadically stable, it is power-Dedekind-finite by Proposition~\ref{prop:nebulous-free-omega-cat}. Therefore by Proposition~\ref{lem:well-orderable-implies-countable}, we have that each $C_n$ is finite. Let $T = \thy(\Cfr)$. $L[T]$ satisfies that for every $r > 0$, $|\ell_n(T)|$ is eventually dominated by $2^{n!r^n}$ by Theorem~\ref{ms_slow} (replacing $c$ with $r^{-1}$). The statement `for every real $r > 0$, $f(n)$ is eventually dominated by $2^{n!r^n}$' is equivalent to the analogous statement quantifying over rational $r$ and is therefore arithmetically definable. It is also clearly downwards closed among functions $f : \omega \to \omega$, so we can apply Proposition~\ref{prop:main-meta-prop} to get that for every real $r > 0$, $|C_n|$ is eventually dominated by $2^{n!r^n}$. 
\end{proof}

\begin{theorem}\label{thm:mon-stable-ZF}
  $(\ZF)$ For any set $A$ and family $(C_n)_{n \in \omega}$ with $C_n \subseteq \Pc(A^{(n)})$, if $\bigcup_{n \in \omega} C_n$ is well-orderable and there is a real $r > 0$ such that $|C_n| > 2^{n!r^n}$ for infinitely many $n \in \omega$, then either
  \begin{itemize}
  \item there is a $D \subseteq A^2$ such that in some forcing extension there are functions $f,g : \mathbb{Q} \to A$ such that for any $i,j \in \mathbb{Q}$, $(f(i),g(j)) \in D$ if and only if $i < j$ or
  \item there is a $D \subseteq A^3$ such that in some forcing extension there are infinite sets $X,Y,Z \subseteq A$ such that $D\cap (X\times Y\times Z)$ is the graph of a bijection $h : X\times Y \to Z$.
  \end{itemize}
\end{theorem}
\begin{proof}
  We will prove the contrapositive. Assume that the conditions in the two bullets both fail. By Fact~\ref{fact:ill-founded-forcing}, this implies that 
  \begin{itemize}
  \item there does not exist a $D \subseteq A^2$ such that the set of pairs of finite partial functions $(f,g)$ from $\mathbb{Q}$ to $A$ with $\mathrm{dom}(f) =\mathrm{dom}(g)$ satisfying the condition in the statement of the theorem is ill-founded under extension, and
  \item there does not exist a $D \subseteq A^{3}$ such that the set $\{(X,Y,Z) \in \Pc_{\mathrm{fin}}(A)^3 : D \cap (X\times Y \times Z)~\text{\rm is the graph of a bijection}~h : X \times Y \to Z\}$ is ill-founded under the strict partial order $(X,Y,Z) < (X',Y',Z')$ defined by $X \subset X' \wedge Y \subset Y' \wedge Z \subset Z'$.
  \end{itemize}
  By Proposition~\ref{prop:nebulous-stability-char}, $A$ is nebulously stable. Therefore by Proposition~\ref{prop:nebulous-free-omega-cat}, $A$ is power-Dedekind-finite and so by Lemma~\ref{lem:well-orderable-implies-countable}, each $C_n$ is finite. By Proposition~\ref{prop:nebulously-not-coding}, $A$ nebulously does not admit coding. By Theorem~\ref{coding},
  we have that $A$ is nebulously monadically stable. The result now follows by Proposition~\ref{prop:neb-mon-stab-rad-conv}.
\end{proof}

We can also use the following proposition to prove a variant of Theorem~\ref{thm:mon-stable-ZF}.

\begin{proposition}\label{prop:NIP-sMf-char}
  $(\ZF)$ If $A$ is nebulously NIP, then $A$ is strictly Mostowski finite if and only if every strict partial order on $A$ is well-founded.
\end{proposition}
\begin{proof}
  Clearly if there is an ill-founded strict partial order $<$ on $A$, then $<$ must have arbitrarily large finite chains and so $A$ is not strictly Mostowski finite, so assume that $A$ is nebulously NIP but not strictly Mostowski finite. Fix a partial order $<$ on $A^n$ with arbitrarily large chains and let $\Afr$ be the structure $(A,<)$. Let $T = \thy(\Afr)$. We have that $T$ is NIP in $L[T]$. By the main theorem of \cite{Simon2021}, there is an unstable formula $\varphi(x,y,\bbar)$ in $T$ for some parameters $\bbar$ in some model of $T$. By Shelah's proof that stability is equivalent to the conjunction of NIP and NSOP, we get that there is a formula $\psi(x,y,\bbar)$ (still with two free variables) defining a strict partial order witnessing that $T$ is SOP.\footnote{See also the discussion after Question~1.3 in \cite{Johnson2025}.} Let $(a_i)_{i \in \omega}$ be an infinite $\psi(x,y,\bbar)$-chain. By Lemma~\ref{lem:omega-cat-ill-founded} applied to the sequence $\bbar,a_0,a_1,a_2,\dots$, we have that there is a $\dbar \in A$ such that $R = \{(x,y) \in A^2 : \Afr\models \psi(x,y,\dbar)\}$ is ill-founded.
\end{proof}

\begin{theorem}\label{thm:mon-stable-ZF-II}
  $(\ZF)$ For any set $A$, if there is a family $(C_n)_{n \in \omega}$ with $C_n \subseteq \Pc(A^{(n)})$ satisfying the assumptions of Theorem~\ref{thm:mon-stable-ZF}, then either
  \begin{itemize}
  \item there is an ill-founded strict partial order on $A$ (and in particular $A$ is not strictly Mostowski-finite) or
    \item there is a $k \in \omega$ and $D \subseteq A^{2k+1}$ such that in some forcing extension there are infinite sets $X,Y \subseteq A^k$ and $Z \subseteq A$ such that $D\cap (X\times Y \times Z)$ is the graph of a bijection $f : X\times Y \to Z$.
  \end{itemize}
\end{theorem}
\begin{proof}
  We will prove the converse. Assume that both bullet points fail. By Fact~\ref{fact:ill-founded-forcing}, the failure of the second bullet point implies that there does not exist an $n \in \omega$ and $D \subseteq A^{2k+1}$ such that the set $\{(X,Y,Z) \in \Pc_{\mathrm{fin}}(A^k) \times \Pc_{\mathrm{fin}}(A^k)\times \Pc_{\mathrm{fin}}(A) : D \cap (X \times Y \times Z)~\text{\rm is the graph of a bijection}~f : X \times Y \to Z\}$ is ill-founded under the strict partial order $(X,Y,Z) < (X',Y',Z')$ defined by $X \subset X' \wedge Y \subset Y' \wedge Z \subset Z'$.
  By Proposition~\ref{prop:nebulously-not-tuple-coding}, $A$ is nebulously monadically NIP and therefore nebulously NIP. By Proposition~\ref{prop:NIP-sMf-char}, this implies that $A$ is strictly Mostowski-finite, and therefore in particular power-Dedekind-finite. By Proposition~\ref{prop:nebulously-not-tuple-coding} again, we have that $A$ nebulously does not admit tuple-coding. It follows from Proposition~\ref{prop:monadic-stability-char} that $A$ is nebulously monadically stable. The result now follows by Proposition~\ref{prop:neb-mon-stab-rad-conv}.
\end{proof}

\subsection{Nebulous cellularity}
\label{sec:nebulous-cellularity}

We can use Lemma~\ref{not_cell} together with the fact that that a countable first-order theory is cellular if and only if it is $\omega$-categorical and mutually algebraic \cite[Thm.~3.7]{braunfeld2022mutual} to get another result in the vein of Theorem~\ref{thm:mon-stable-ZF}.

\begin{theorem}\label{thm:cellular-ZF}
  $(\ZF)$ For any set $A$ and family $(C_n)_{n \in \omega}$ with $C_n \subseteq \Pc(A^{(n)})$, if $\bigcup_{n \in \omega} C_n$ is well-orderable and for any $c,d \in \R$ with $d < 1$, there is an $n$ such that $|C_n| > 2^{cn^{dn}}$, then either
  \begin{itemize}
  \item there is a $D \subseteq A^2$ such that in some forcing extension there are functions $f,g : \mathbb{Q} \to A$ such that for any $i,j \in \mathbb{Q}$, $(f(i),g(j)) \in D$ if and only if $i < j$,
  \item there is a $D \subseteq A^3$ such that in some forcing extension there are infinite sets $X,Y,Z \subseteq A$ such that $D\cap (X\times Y\times Z)$ is the graph of a bijection $f : X\times Y \to Z$, or
  \item $A$ can be partitioned into infinitely many infinite sets.
  \end{itemize}
\end{theorem}
\begin{proof}
  Assume that the first two bullet points fail. By the same argument as in the proof of Theorem~\ref{thm:mon-stable-ZF}, $A$ is nebulously monadically stable and $\omega$-categorical. Let $T = \thy(\Cfr)$. By \cite[Lem.~3.49]{bodor2024classification}, we must have that $T$ is not cellular in $L[T]$. By \cite[Thm.~3.7]{braunfeld2022mutual}, $T$ is not mutually algebraic in $L[T]$. Therefore we must have by  Lemma~\ref{not_cell} that there in the unique countable model of $T$ there is a definable equivalence relation $\varphi(x,y,b)$ with infinitely many infinite equivalence classes. By picking some $c$ in the structure on $A$ induced by $\Cfr$ with the same type as $b$, we get that $E = \{(x,y) \in A^2 : \varphi(x,y,c)\}$ is also an equivalence relation with infinitely many infinite equivalence classes. By taking combining the finite $E$-classes with some infinite $E$-class, we get a partition of $A$ into infinitely many infinite sets.
\end{proof}

Of course, we can also replace the first two bullet points of Theorem~\ref{thm:cellular-ZF} with the conclusions of Theorem~\ref{thm:mon-stable-ZF-II}. 

\subsection{A structure theorem for amorphous sets}
\label{sec:structure-theorem-amorphous}

Here we will give a structural dichotomy for amorphous sets, strengthening some of the results of \cite{Truss_1995}.

\begin{definition}
  For any vector space $X$, let $\mathbf{P}(X)$ be the projective space of $X$.

  Fix a finite field $\mathbb{F}_q$. The \emph{(countably) infinite-dimensional projective space over $\mathbb{F}_q$} is $\mathbf{P}(\mathbb{F}_q^{\bigoplus \omega})$ endowed with ternary collinearity relation and a function for each $\alpha \in \mathbb{F}_q$ that take collinear distinct $x_0,x_1,x_2$ and returns the unique $x_3$ such that the cross ratio $[x_0,x_1,x_2,x_3]$ is $\alpha$.

  An \emph{infinite-dimensional projective space structure (over $\mathbb{F}_q$)} is a model of the above theory.
\end{definition}

The automorphism group of the infinite-dimensional projective space (over $\mathbb{F}_q$) in the sense of the first-order structure given here is the same as it is in the typical sense (i.e., $\mathrm{PGL}(\omega,\mathbb{F}_q)$) \cite[p.~121]{Pillay1996-lt}. Note, however, that in the absence of choice, we cannot show that if a set $A$ admits an infinite-dimensional projective space structure over $\mathbb{F}_q$, then this structure is isomorphic to $\mathbf{P}(X)$ for some $\mathbb{F}_q$-vector space $X$.

\begin{proposition}\label{prop:omega-cat-strongly-minimal-char}
  $(\ZFC)$ Fix a countable complete $\omega$-categorical strongly minimal theory $T$. The following are equivalent.
  \begin{enumerate}
  \item\label{non-trivial-forking} $T$ has non-trivial forking.
  \item\label{affine-or-projective-geometry} The geometry of $T$ is either an infinite-dimensional affine space over a finite field or an infinite-dimensional projective space over a finite field.
  \item \label{interprets-group} For any $M \models T$, $T$ interprets an infinite group using parameters from $M$.
  \item\label{specific-automorphism-group} For any $M \models T$, $T$ has a one-dimensional interpretation using parameters from $M$ of an infinite-dimensional projective space structure over a finite field.
  \item\label{T-not-mod-stab} $T$ is not monadically stable.
  \item $T$ is not monadically NIP.
  \item $T$ is not mutually algebraic.
  \item\label{T-not-cellular} $T$ is not cellular.
  \item\label{Bell-equivalent} $s_n(T) > {B_n^{(2)}}$ for infinitely many $n$.
  \item\label{power-equivalent} There is a quadratic polynomial $p(n)$ with leading term $\frac{1}{4}n^2$ such that $\ell_n(T) > 2^{p(n)}$ for all $n$.
  \end{enumerate}
\end{proposition}
\begin{proof}
  Let (\ref{interprets-group})$^\ast$ and (\ref{specific-automorphism-group})$^\ast$ be the analogous statements with `For some $M \models T$' instead of `For any $M\models T$'.
  
  The equivalence of (\ref{non-trivial-forking}), (\ref{affine-or-projective-geometry}), (\ref{interprets-group})$^\ast$,  (\ref{specific-automorphism-group})$^\ast$, and (\ref{T-not-mod-stab})--(\ref{T-not-cellular}) follows from the fact that interpreting an infinite group entails non-triviality of forking together with \cite[Thm.~2.1]{Cherlin_1985}, the group configuration theorem, \cite[Thm.~3.7]{braunfeld2022mutual}, and Theorem~\ref{coding}.\footnote{See \cite[Ch.~2, Sec.~4]{Pillay1996-lt} for a specific discussion of the group configuration in the case of $\omega$-categorical strongly minimal theories, including the fact that it gives a definable group instead of merely a type-definable group, and see \cite[p.~119]{Pillay1996-lt} for an argument that (\ref{specific-automorphism-group}) holds when the geometry of $T$ is affine or projective.} Since $T$ is $\omega$-stable, we get that (\ref{interprets-group})$^\ast$ and (\ref{specific-automorphism-group})$^\ast$ are equivalent to (\ref{interprets-group}) and (\ref{specific-automorphism-group}), respectively, so we have that (\ref{non-trivial-forking})-(\ref{T-not-cellular}) are equivalent.

  Assume (\ref{affine-or-projective-geometry}). By \cite[Thm.~1.4]{Macpherson_1987} (and specifically the proof of \cite[Lem.~3.4]{Macpherson_1987}), there is a polynomial $p(n)$ of degree at least $2$ (with positive leading coefficient) such that $\ell_n(T) > 2^{p(n)}$ for all $n$. By Fact~\ref{fact:Steitz} (and since $\ell_n(T) \leq s_n(T)$), $p(n)$ must be quadratic and so (\ref{power-equivalent}) holds.

  Clearly, $B_n^{(2)}\leq B_n^2$, thus $\log B_n^{(2)}$ is $\mathcal{O}(n\log n)$. By this it follows that (\ref{power-equivalent}) implies (\ref{Bell-equivalent}).

  Finally, we will prove that $\neg$(\ref{T-not-cellular}) implies $\neg$(\ref{Bell-equivalent}). Assume that (\ref{T-not-cellular}) fails. Then by Theorem~\ref{cd} and Remark~\ref{cd_bell} we know that $\oi_n(T)\leq B_n$ if $n$ is large enough. This implies that $s_n(T) \leq \sum_{k=1}^n \bracenom{n}{k}B_k=B_n^{(2)}$.
\end{proof}

Note that in the following theorem, (\ref{amorphous-projective-space}) is equivalent to the statement that $A$ is of projective type in the sense of \cite{Truss_1995}.

\begin{theorem}\label{thm:amorphous-structure-theorem}
  $(\ZF)$ For any amorphous set $A$, the following are equivalent.
  \begin{enumerate}
  \item\label{amorphous-group} There is a $k\in \omega$ such that some infinite quotient of $A^k$ is the underlying set of a group.
  \item\label{amorphous-projective-space} Some quotient of $A$ is the underlying set of an infinite-dimensional projective space structure over a finite field.
  \item\label{amorphous-forcing} There is a $D \subseteq A^3$ such that in some forcing extension there are infinite sets $X,Y,Z \subseteq A$ such that $D\cap (X\times Y\times Z)$ is the graph of a bijection $h : X\times Y \to Z$.
  \item\label{amorphous-finitary} There is a $k \in \omega$ and $D \subseteq A^{2k+1}$ such that for every $m \in \omega$, there are $X,Y \subseteq A^k$ and $Z \subseteq A$ such that $|X| = |Y| = m$ and $D \cap (X\times Y\times Z)$ is the graph of a bijection $h : X \times Y \to Z$.
  \item\label{amorphous-slow} There is a family $(C_n)_{n \in \omega}$ with $C_n \subseteq \Pc(A^n)$ such that $\bigcup_{n \in \omega}C_n$ is well-orderable and $|C_n| > 2^{B_n^{(2)}}$ for infinitely many $n$.
  \item\label{amorphous-faster} There is a family $(C_n)_{n \in \omega}$ with $C_n \subseteq \Pc(A^{(n)})$ such that $\bigcup_{n \in \omega}C_n$ is well-orderable and for a quadratic polynomial $p(n)$ with leading term $\frac{1}{4}n^2$, $|C_n| > 2^{2^{p(n)}}$ for all $n$. 
  \end{enumerate}
\end{theorem}
\begin{proof}
  By Fact~\ref{fact:WalczakTypkeCorrespondence}, any theory supported by $A$ must be strongly minimal and $\omega$-categorical, so the equivalence of (\ref{amorphous-group})-(\ref{amorphous-faster}) follows by Propositions~\ref{prop:nebulously-not-coding} and \ref{prop:omega-cat-strongly-minimal-char}.
\end{proof}

\bibliographystyle{alpha}
\bibliography{ref.bib}

\end{document}